\documentclass[12pt]{elsarticle}

\usepackage{amsmath,amsfonts,amssymb}
\usepackage{mathtools}
\usepackage{mathdots}
\usepackage{graphicx}
\usepackage[colorlinks,linkcolor=blue,citecolor=blue]{hyperref}
\usepackage{color}

\usepackage{amssymb,amsmath,amsfonts}
\usepackage{color}
\usepackage{amsthm}

\newcommand{\QQ}{\ensuremath{\mathbb{Q}}}
\newtheorem{theorem}{Theorem}

\newtheorem{proposition}[theorem]{Proposition}

\theoremstyle{definition}

\newtheorem{example}[theorem]{Example}

\begin{document}

\begin{frontmatter}

\title{Galois groups and rational solutions of $p(X) = A$.}

\author[label1]{G.J.~Groenewald} 
\author[label1]{G. Goosen}
\author[label1]{D.B.~Janse van Rensburg}
\author[label2]{A.C.M.~Ran}
\author[label1]{M.~van~Straaten}

\address[label1]{School~of~Mathematical~and~Statistical~Sciences,
North-West~University,
Research Focus: Pure and Applied Analytics,
Private~Bag~X6001,
Potchefstroom~2520,
South Africa.
E-mail: \texttt{gilbert.groenewald@nwu.ac.za, gerrit.goosen@nwu.ac.za, dawie.jansevanrensburg@nwu.ac.za, madelein.vanstraaten@nwu.ac.za}}
\address[label2]{Department of Mathematics, Faculty of Science, VU Amsterdam, De Boelelaan
    1111, 1081 HV Amsterdam, The Netherlands
    and Research Focus: Pure and Applied Analytics, North-West~University,
Potchefstroom,
South Africa. E-mail:
    \texttt{a.c.m.ran@vu.nl}}

\begin{abstract}
We extend Theorem 1 of R. Reams, A Galois approach to m-th roots of matrices with rational entries,
LAA 258 (1997), 187-194. Let $p(\lambda)$ be any polynomial over $\mathbb{Q}$ and let $A\in M_n(\mathbb{Q})$
have irreducible characteristic polynomial $f(\lambda)$ with degree n. We provide necessary and sufficient conditions
for the existence of a solution $X\in M_n(\mathbb{Q})$ of the polynomial matrix equation $p(X) = A.$ Specifically, we
find necessary and sufficient conditions for $f(p(\lambda))$ to have a factor of degree $n$ over $\mathbb{Q}.$
\end{abstract}


\end{frontmatter}

{\date{}}

\section{Introduction} \label{sec:intro}

\bigskip

Let $A$ be an $n \times n$ matrix with rational entries, let $f(\lambda)$ denote the characteristic polynomial of $A$, and suppose further that $f(\lambda)$ is irreducible. Let $p(\lambda)$ be a polynomial of degree $m$ with rational coefficients. We are interested in finding rational solutions $X$ of the matrix equation $p(X) = A$.

This problem has been considered in the literature before. In \cite{reams} a purely algebraic approach is taken for the special case $p(\lambda)=\lambda^m$ with $m$ odd. A more constructive approach for general $p(\lambda)$ was taken in \cite{Drazin}. In \cite{GJRST2} the condition that $f(\lambda)$ is irreducible is relaxed to $A$ being nonderogatory, following a constructive approach very much related to the one in \cite{Drazin}. Another, more numerical approach can be found in \cite{FI1, FI2, FI3}. Allowing for any complex solution, the problem of finding solutions to $p(X)=A$, with $p(\lambda)$ being a holomorphic function, was considered in \cite{EU}, see also \cite{roth, Spiegel}.

The special case where $p(\lambda)=\lambda^m$, i.e. the case where $X$ is an $m$th root of $A$ has been studied in detail in several papers; see for example \cite{higham,otero, psarrakos, tenhave, wedderburn}. The case where the additional symmetry of $H$-selfadjointness is involved, is treated in \cite{GJRST}. 

Reams shows in \cite{reams} that $p(X)=A$ has a rational solution if and only if $f(p(\lambda))$ has a factor $h(\lambda)$ of degree $n$ in $\mathbb{Q}[\lambda]$; he attributes this result to previously unpublished work by Tom Laffey and Bryan Cain. He then continues with a purely algebraic condition for the case where $p(X)=X^m$ with $m$ odd, in terms of the Galois groups of $f(\lambda)$ and $f(\lambda^m)$. 
We shall expand on the first equivalence stated at the beginning of this paragraph, and connect it to a condition found in \cite{Drazin}, which results in an explicit construction of the factor $h(\lambda)$. The focus in the present article is on purely algebraic conditions. In \cite{GJRST2} (Theorem 7.1) the equation $p(X)=A$ is considered using a predominantly linear algebraic approach, leading to a result for nonderogatory matrices $A$. We shall restate that result in an alternative purely algebraic way in the final section of this paper, although the result can be obtained directly from our earlier paper \cite{GJRST2}.

\bigskip

\section{Preliminaries} \label{sec:prelim}

We shall use the following notation: the eigenvalues of $A$ (and hence the roots of $f(\lambda)$) are denoted by $\mu_i$, $i=1, \ldots , n$, and in case we have $h(\lambda)$ given, then the roots of $h(\lambda)$ are denoted by $\gamma_i$, $i=1, \ldots , n$. 

We recall in this section several definitions and results for the convenience of the reader. Notation will conform with usual practices in abstract algebra, see, e.g., \cite{pinter}.
In particular, we will adopt the convention of referring to roots (or zeroes)
of a polynomial $f(\lambda)$ which will also be taken to mean solutions of the polynomial
equation $f(\lambda)=0.$

\bigskip

 Let $f(\lambda)$ be a polynomial of degree $n$ over $\QQ$. Write

\[
    f(\lambda) = (\lambda - \mu_1)(\lambda - \mu_2) \cdots (\lambda - \mu_n)
\]

\noindent where $\mu_j \in \overline{\QQ}$ for each $j$, with $\overline{\QQ}$ denoting the algebraic closure of $\QQ$. The \emph{Galois group of $f(\lambda)$} over $\mathbb{Q}$ is defined to be the group of automorphisms of $\QQ(\mu_1,\cdots,\mu_n)$ which fix $\QQ$. We denote this group by $\mathit{Gal}(\QQ(\mu_1,\cdots,\mu_n):\QQ)$ or simply by $G$ when the context is clear.\\

\noindent In the sequel we will require our Galois group $G$ to act on a certain set of roots. We therefore state the necessary definitions.

Let $G$ be a group and $X$ a set. A \textit{group action} of $G$ on $X$ is a map 
\[
    \cdot : G \times X \rightarrow X,
\]
where we write $g \cdot x$ for the value of the map on the pair $(g,x)$, satisfying the following properties:
\begin{itemize}
    \item[(i)] $e \cdot x = x$ for all $x \in X$,
    \item[(ii)] $g \cdot (h \cdot x) = (gh) \cdot x$ for all $g,h \in G$ and all $x \in X$.
\end{itemize}

A group $G$ is said to act \textit{transitively} on $X$ if for all $x,y \in X$ there exists an element $g \in G$ such that $g \cdot x = y$.

The following result can be found in \cite{Stewart}, Proposition 22.3, or \cite{Bewersdorff}, Section 10.10.

\begin{theorem} \label{Thm:irr+GroupAction}
Let $f(\lambda)\in\mathbb{Q}[\lambda]$ be irreducible and monic. Let $X=\{\mu_1,\ldots,\mu_n\}$ be the roots of $f(\lambda)$, which lie in some fixed algebraic closure of $\mathbb{Q}$. Let $G$ be the Galois group of $f(\lambda)$. Then $G$ acts transitively on $X$ via the action
\begin{equation*}
g \cdot \mu = g(\mu),
\end{equation*} 
where $g \in G$, $\mu \in X$.
\end{theorem}

\bigskip

We recall
here the spectral mapping theorem for the finite dimensional case. For a proof, see e.g., \cite{roman}, Theorem 8.3.

\begin{theorem} 
    Let $V$ be a finite dimensional vector space over an algebraically closed field $\mathbb{F}$. Let $T : V \to V$ be a linear map with spectrum $\sigma(T)$, and let $p(\lambda)\in \mathbb{F}[\lambda]$. Then 
    \[
        \sigma(p(T)) = p(\sigma(T)) = \{p(\mu) \; | \; \mu \in \sigma(T)\}.
    \]
\end{theorem}

The following proposition is a restatement of \cite{Drazin}, Proposition 2.3, for the casein which  we are interested.

\begin{proposition}
Let $A, X\in M_n(\mathbb{Q})$ and $p(\lambda)\in\mathbb{Q}[\lambda]$, and suppose that the characteristic polynomial of $A$, denoted by $f(\lambda)$ is irreducible and that $A=p(X)$. \\
Then, for each eigenvalue $\mu_j$ of $A$, $j=1, \hdots, n$, the equation $p(\gamma)=\mu_j$ has at least one solution $\gamma=\gamma_j\in\mathbb{Q}(\mu_j)$.
\end{proposition}

\bigskip

\section{Main Results} \label{sec:main}

The main theorem that we prove in this article is an extension of Theorem 1 of \cite{reams} and is stated as follows.

\begin{theorem}\label{Thm:main}
    Let $n$ be a natural number, $p(\lambda)$ any polynomial over $\mathbb{Q}$ and let $A \in M_{n}(\mathbb{Q})$ have irreducible characteristic polynomial $f(\lambda)$ with degree $n$. Let $\mu_i, 1 \leq i \leq n$, denote the roots of $f(\lambda)$. Then the following are equivalent:
    
    \begin{itemize}
        \item[(i)] $A = p(X)$ has a solution over $\mathbb{Q}$;
        \item[(ii)]  $f(p(\lambda))$ has a factor $h(\lambda)$ of degree $n$ over $\mathbb{Q}$;
        \item[(iii)] There exist an eigenvalue $\mu\in\sigma(A)$ and an element 
        $\gamma\in\mathbb{Q}(\mu)$ such that $p(\gamma) = \mu$.
    \end{itemize}
\end{theorem}

It will follow from the proof  that the third statement above is also equivalent to: for every eigenvalue $\mu\in \sigma(A)$ there is an element $\gamma\in\mathbb{Q}(\mu)$ such that $p(\gamma) = \mu$.

The implication (i) implies (iii) follows from \cite{Drazin}, Proposition 2.3. 
We will provide an alternative independent argument in the proof below.

\begin{proof}
The equivalence of (i) and (ii) is already stated as Proposition 1, \cite{reams}, which attributes the result to T.J. Laffey and B. Cain. For completeness' sake, we provide the main ideas of the proof. Assuming (i) holds, let $X$ be a rational solution of $p(X)=A$, let $h(\lambda)$ be the minimal polynomial of $X$, and let $\gamma_i, i=1, \ldots , n$ be the eigenvalues of $X$. By the spectral mapping theorem $p(\gamma_i)=\mu_i$ (after possibly reordering), and since $f(\lambda)$ is irreducible, this means that the $\gamma_i$'s are all different, as are the $\mu_i$'s. Hence $h(\lambda)$ is a polynomial of degree $n$ in $\mathbb{Q}[\lambda]$. Moreover, by the Cayley-Hamilton theorem $f(A)=f(p(X))=0$, and hence $h(\lambda)$ divides $f(p(\lambda))$. 

Conversely, suppose $h(\lambda)$ is a polynomial of degree $n$ in $\mathbb{Q}[\lambda]$ which divides $f(p(\lambda))$. Let $C_h$ denote the companion matrix of $h(\lambda)$. Then $p(C_h)$ is similar to $A$, because $f(p(C_h))=0$ and $f(\lambda)$ is irreducible. Hence, there is an invertible matrix $S$ such that $S^{-1}p(C_h)S=A$. Take $X=S^{-1}C_hS$, then $p(X)=A$.

\medskip

(ii) $\Rightarrow$ (iii). Let $h(\lambda)\in \mathbb{Q}[\lambda]$ be a factor of degree $n$ of $f(p(\lambda))$, and let us say $h(\lambda) = (\lambda -\gamma_1) \cdots (\lambda - \gamma_n)$, where $\gamma_{i} \in \overline{\mathbb{Q}}$, for $1 \leq i \leq n.$

Then $f(p(\gamma_i)) = 0$, so that $p(\gamma_{i}) = \mu_{j_i}$ for some $j_i \in \{1,2, \ldots , n \}$. Hence $\mu_{j_i}\in\mathbb{Q}(\gamma_i)$.

Since $n\geq \left[\mathbb{Q}(\gamma_i):\mathbb{Q} \right] = \left[\mathbb{Q}(\gamma_i):\mathbb{Q}(\mu_{j_i}) \right] \cdot \left[\mathbb{Q}(\mu_{j_i}):\mathbb{Q} \right]$ and $\left[\mathbb{Q}(\mu_{j_i}):\mathbb{Q} \right] = n$, it follows that $\left[\mathbb{Q}(\gamma_i):\mathbb{Q} \right] = n$, so $h(\lambda) \in \mathbb{Q}[\lambda]$ must be irreducible.
By the same argument, $ \left[\mathbb{Q}(\gamma_i):\mathbb{Q}(\mu_{j_i}) \right]=1$, and hence $\gamma_i\in\mathbb{Q}(\mu_{j_i})$, so that in fact $\mathbb{Q}(\gamma_i)=\mathbb{Q}(\mu_{j_i})$.
This shows part (iii).

\medskip

(iii) $\Rightarrow$ (ii). Let $\mu$ be an eigenvalue of $A$ such that there is a $\gamma\in\mathbb{Q}(\mu)$ with $p(\gamma)=\mu$. Without loss of generality, possibly after renumbering the $\mu_i$'s, we may assume that this holds for $\mu_1$, and let us denote $\gamma$ by $\gamma_1$.

Since $f(\lambda)$ is irreducible, the Galois group $G$ of $f(\lambda)$ acts transitively on $\{\mu_1,\cdots,\mu_n\}$. That is, by Theorem \ref{Thm:irr+GroupAction}, for each $\mu_j \in \sigma(A)$ there is an automorphism $g_j$ in $G$ such that $g_j(\mu_1)=\mu_j$. Define $\gamma_j=g_j(\gamma_1)$, and take 
\[
h(\lambda)=(\lambda-\gamma_1)(\lambda-\gamma_2)\cdots (\lambda-\gamma_n).
\]
Then $p(\gamma_j)=p(g_j(\gamma_1))=g_j(p(\gamma_1))=g_j(\mu_1)=\mu_j$, because $p(\lambda)$ has rational coefficients, and $g_j$ fixes $\mathbb{Q}$. Hence $f(p(\gamma_j))=f(\mu_j)=0$, and since the $\mu_j$'s are all different, so are the $\gamma_j$'s. Thus, $h(\lambda)$ is a divisor of $f(p(\lambda))$ of degree $n$. 

It remains to show that $h(\lambda)\in\mathbb{Q}[\lambda]$. We shall show this in a  direct manner.  The proof is based on the fact that
\[
h(\lambda)=\sum_{j=0}^n \lambda^{n-j} e_{j}(\gamma_1, \ldots, \gamma_n),
\]
where $e_j(\gamma_1,\ldots , \gamma_n)$ is the elementary symmetric polynomial of degree $j$. So,
\begin{align*}
e_0(\gamma_1, \ldots, \gamma_n)&=1, \\ e_1(\gamma_1, \ldots, \gamma_n)&=\gamma_1+\cdots +\gamma_n,\\
e_2(\gamma_1, \ldots, \gamma_n)&=\sum_{1\leq k<j\leq n} \gamma_k\gamma_j,\\
e_3(\gamma_1, \ldots, \gamma_n)&=\sum_{1\leq k<j<l\leq n} \gamma_k\gamma_j\gamma_l,\\
\vdots & \\
e_n(\gamma_1, \ldots, \gamma_n)&=\Pi_{j=1}^n \gamma_j.
\end{align*}
In a similar manner
\[
f(\lambda)=(\lambda-\mu_1)\cdots (\lambda-\mu_n)=\sum_{j=0}^n \lambda^{n-j} e_{j}(\mu_1, \ldots, \mu_n),
\]
and since we know that $f(\lambda)\in\mathbb{Q}[\lambda]$, we have that $e_j(\mu_1, \ldots, \mu_n)\in\mathbb{Q}$ for $j=1, \ldots , n$. 

Now, because of the fact that $\gamma_1\in\mathbb{Q}(\mu_1)$, and $f(\lambda)$ is the minimal polynomial of $\mu_1$ over $\mathbb{Q}$, there are rational numbers
$\alpha_0, \ldots , \alpha_{n-1}$ such that 
\[
\gamma_1=\alpha_0+\alpha_1\mu_1+\alpha_2\mu_1^2+\cdots + \alpha_{n-1}\mu_1^{n-1}.
\]
Applying $g_i$ left and right we obtain
\[
\gamma_i=\alpha_0+\alpha_1\mu_i+\alpha_2\mu_i^2+\cdots + \alpha_{n-1}\mu_i^{n-1},
\]
for the same $\alpha_0, \ldots , \alpha_{n-1}$. We have to show that $e_j(\gamma_1, \ldots ,\gamma_n)\in\mathbb{Q}$
for $j=0, 1, \ldots , n$. To see this, note that by inserting the formulas for $\gamma_i $ in terms of $\mu_i$, we have
\[
e_j(\gamma_1,\gamma_2 \ldots ,\gamma_n)=e_j\left(\sum_{j=0}^{n-1} \alpha_j\mu_1^j, \sum_{j=0}^{n-1} \alpha_j\mu_2^j, \ldots , \sum_{j=0}^{n-1} \alpha_j\mu_n^j\right).
\]
This is a symmetric polynomial in $\mu_1, \ldots , \mu_n$ with rational coefficients. By the \emph{fundamental theorem of symmetric polynomials} (see \cite{Bewersdorff}, Theorem 5.1) any symmetric polynomial in $\mu_1, \ldots , \mu_n$ with rational coefficients has a unique representation as a polynomial in $e_1(\mu_1, \ldots , \mu_n), \ldots , e_n(\mu_1, \ldots , \mu_n)$ with rational coefficients. Since the numbers
$e_j(\mu_1, \ldots , \mu_n)$ are rational as well, it follows that also the numbers $e_j(\gamma_1, \ldots , \gamma_n)$ are rational. 
\end{proof}

The following example illustrates in a concrete manner the final part of the above proof.

\begin{example} 
Take $n=3$ (for $n=2$ it is all fairly easy). We have to show that
$\gamma_1+\gamma_2+\gamma_3\in\mathbb{Q}$, $\gamma_1\gamma_2+\gamma_1\gamma_3+\gamma_2\gamma_3\in\mathbb{Q}$, and $\gamma_1\gamma_2\gamma_3\in\mathbb{Q}$. 
We use $\gamma_j=\alpha_0+\alpha_1\mu_j+\alpha_2\mu_j^2$ for $j=1,2,3$.
Then
\[
\gamma_1+\gamma_2+\gamma_3=3\alpha_0+\alpha_1(\mu_1+\mu_2+\mu_3)+
\alpha_2(\mu_1^2+\mu_2^2+\mu_3^2).
\]
We know that $\mu_1+\mu_2+\mu_3\in\mathbb{Q}$, so to show that $\gamma_1+\gamma_2+\gamma_3\in\mathbb{Q}$ it suffices to show that $\mu_1^2+\mu_2^2+\mu_3^2\in\mathbb{Q}$.
To see the latter we could invoke the fundamental theorem of symmetric polynomials, but let us do it directly:
\[
(\mu_1+\mu_2+\mu_3)^2=\mu_1^2+\mu_2^2+\mu_3^2 +2(\mu_1\mu_2 +\mu_1\mu_3+\mu_2\mu_3).
\]
So we have
\[
\mu_1^2+\mu_2^2+\mu_3^2 =e_1(\mu_1,\mu_2,\mu_3)^2-2e_2(\mu_1,\mu_2,\mu_3).
\]
The right-hand side is rational, so the left-hand side is rational as well.

Next, we compute $\gamma_1\gamma_2+\gamma_1\gamma_3+\gamma_2\gamma_3$. After a bit of computation we obtain
\begin{align*}
&\gamma_1\gamma_2+\gamma_1\gamma_3+\gamma_2\gamma_3=
3\alpha_0^2+2\alpha_0\alpha_1(\mu_1+\mu_2+\mu_3) 
\\
&+2\alpha_0\alpha_2(\mu_1^2+\mu_2^2+\mu_3^2)+\alpha_1^2(\mu_1\mu_2+\mu_1\mu_3+\mu_2\mu_3) \\
&+\alpha_1\alpha_2(\mu_1^2\mu_2+\mu_1\mu_2^2+\mu_1^2\mu_3+\mu_1\mu_3^2+\mu_2^2\mu_3+\mu_2\mu_3^2) \\
&+\alpha_2^2((\mu_1\mu_2)^2+(\mu_1\mu_3)^2+(\mu_2\mu_3)^2).
\end{align*}
We already have that $\mu_1^2+\mu_2^2+\mu_3^2$ is rational, thus we only need to show that the last two terms are rational. For the term with $\alpha_1\alpha_2$ use
\begin{align*}
&\mu_1^2\mu_2+\mu_1\mu_2^2+\mu_1^2\mu_3+\mu_1\mu_3^2+\mu_2^2\mu_3+\mu_2\mu_3^2
\\ =&
(\mu_1\mu_2+\mu_1\mu_3+\mu_2\mu_3)(\mu_1+\mu_2+\mu_3)-3\mu_1\mu_2\mu_3\\ =&
e_2(\mu_1,\mu_2,\mu_3)e_1(\mu_1,\mu_2,\mu_3)-3e_3(\mu_1,\mu_2,\mu_3).
\end{align*}
For the term with $\alpha_2^2$, use
\begin{align*}
&(\mu_1\mu_2)^2+(\mu_1\mu_3)^2+(\mu_2\mu_3)^2\\ =&
(\mu_1\mu_2+\mu_1\mu_3+\mu_2\mu_3)^2-2(\mu_1^2\mu_2\mu_3+\mu_1\mu_2^2\mu_3+\mu_1\mu_2\mu_3^2) \\ =&
(\mu_1\mu_2+\mu_1\mu_3+\mu_2\mu_3)^2 -2(\mu_1+\mu_2+\mu_3)\mu_1\mu_2\mu_3 \\ =&
e_2(\mu_1,\mu_2,\mu_3)^2-2e_1(\mu_1,\mu_2,\mu_3)e_3(\mu_1,\mu_2,\mu_3).
\end{align*}

Finally, we compute $\gamma_1\gamma_2\gamma_3$. This gives a total of 27 terms, which we group as follows
\begin{align*}
\gamma_1\gamma_2\gamma_3=&(\alpha_0+\alpha_1\mu_1+\alpha_2\mu_1^2)(\alpha_0+\alpha_1\mu_2+\alpha_2\mu_2^2)(\alpha_0+\alpha_1\mu_3+\alpha_2\mu_3^2)\\ =&
\alpha_0^3+\alpha_0^2\alpha_1(\mu_1+\mu_2+\mu_3)+\alpha_0^2\alpha_2(\mu_1^2+\mu_2^2+\mu_3^2) \\ 
&+\alpha_1^2\alpha_0(\mu_1\mu_2+\mu_1\mu_3+\mu_2\mu_3) \\
&+\alpha_0\alpha_2^2((\mu_1\mu_2)^2+(\mu_1\mu_3)^2+(\mu_2\mu_3)^2) \\
&+\alpha_0\alpha_1\alpha_2(\mu_1^2\mu_2+\mu_1^2\mu_3+\mu_1\mu_2^2+\mu_1\mu_3^2 +\mu_2^2\mu_3+\mu_2\mu_3^2) \\
&+\alpha_1^3\mu_1\mu_2\mu_3+\alpha_1^2\alpha_2(\mu_1\mu_2\mu_3^2+\mu_1\mu_2^2\mu_3+\mu_1^2\mu_2\mu_3)\\
&+\alpha_1\alpha_2^2(\mu_1\mu_2^2\mu_3^2+\mu_1^2\mu_2\mu_3^2+\mu_1^2\mu_2^2\mu_3) +\alpha_2^2(\mu_1\mu_2\mu_3)^2.
\end{align*}
From the above, all terms are in $\mathbb{Q}$ except for the terms  involving $\alpha_1^2\alpha_2$ and $\alpha_1\alpha_2^2$. 
For the term with $\alpha_1^2\alpha_2$, use 
\[
\mu_1\mu_2\mu_3^2+\mu_1\mu_2^2\mu_3+\mu_1^2\mu_2\mu_3=
\mu_1\mu_2\mu_3(\mu_1+\mu_2+\mu_3)=e_3(\mu_1,\mu_2,\mu_3)   e_1(\mu_1,\mu_2,\mu_3).
\]
For the term with $\alpha_1\alpha_2^2$, use
\begin{align*}
\mu_1\mu_2^2\mu_3^2+\mu_1^2\mu_2\mu_3^2+\mu_1^2\mu_2^2\mu_3 &=
\mu_1\mu_2\mu_3(\mu_1\mu_2 +\mu_1\mu_3+\mu_2\mu_3) \\ &=e_3(\mu_1,\mu_2,\mu_3) e_2(\mu_1,\mu_2,\mu_3).
\end{align*}

It follows that $h(\lambda)=(\lambda-\gamma_1)(\lambda-\gamma_2)(\lambda-\gamma_3)$ has rational coefficients.
Obviously, while this brute-force approach works for $n=2$ and $n=3$, it will be very hard to do for $n>3$, so an appeal to the fundamental theorem of symmetric polynomials is very much preferred for the general case.
\hfill$\Box$
\end{example}

\bigskip

In \cite{reams}, for the special case $p(\lambda)=\lambda^m$ with $m>2$, a connection is made between the solvability of $X^m=A$ and the orders of the Galois groups of $f(\lambda^m)$ and $f(\lambda)$. Now, the fact that (i) is equivalent to (ii) in Theorem \ref{Thm:main} does not require the condition $\mathbb{Q}(\lambda_1, \ldots , \lambda_n)\cap \mathbb{Q}(\zeta)=\mathbb{Q}$ for some choice of roots $\lambda_1, \ldots , \lambda_n$ of $f(\lambda^2)$, as stated in Theorem 1 of \cite{reams}. Moreover, Theorem \ref{Thm:main} does hold when $m=2$, since if we consider $f(\lambda)=\lambda^3+3$, then the following example shows that $f(\lambda^2)=\lambda^6+3$ has no factor $h(\lambda)$ of degree $3$ in $\mathbb{Q}[\lambda]$. This is in contrast with the main theorem in \cite{reams}, which does not hold in case $m=2$.

\begin{example}
Consider
\[
A=\begin{bmatrix*}[r] 0 & 1 & 0 \\ 0 & 0 & 1 \\ -3 & 0 & 0  \end{bmatrix*}.
\]
The matrix $A$ has characteristic polynomial $f(\lambda)=\lambda^3+3$. We take $p(\lambda)=\lambda^2$. The roots of $f(p(\lambda))=\lambda^6+3$ are given by
$\gamma_k=\sqrt[6]{3} e^{i(\pi/6+k\pi/3)}$ for $k=0, 1, \ldots, 5$. The minimum polynomial of each of these roots is of degree $6$, actually, since $f(p(\lambda))$ is irreducible over $\mathbb{Q}$ this is the minimum polynomial for each of these roots. Hence for each of these roots we have $[\mathbb{Q}(\gamma_k):\mathbb{Q}]=6$. As we see from the proof of Theorem~\ref{Thm:main}, this would have to be $3$  for the existence of a rational matrix $X$ such that $X^2=A$. It is also directly clear that 
none of the $\gamma_k$'s are in $\mathbb{Q}(\sqrt[3]{3})$ or in $\mathbb{Q}(\sqrt[3]{3}e^{i\pi/3})$, because that would imply that the minimal polynomial of such a $\gamma_k$ would have degree $3$ rather than $6$.\hfill $\Box$
\end{example}

\bigskip 
 
To illustrate the fact that the connection made in \cite{reams} 
between the solvability of $X^m=A$ and the orders of the Galois groups of $f(\lambda^m)$ and $f(\lambda)$
is very specific for the case $p(\lambda)=\lambda^m$ with $m>2$, we consider in the next example the Galois groups for a quadratic polynomial.

\begin{example}
    Let $p(\lambda)=\lambda^2-\lambda-1$ and let $f(\lambda)=\lambda^3+3\lambda^2+21\lambda-11$. Then $f(\lambda)$ is irreducible over $\mathbb{Q}$. Note that this example was constructed by taking a matrix 
    \begin{equation*}
        X =
        \begin{bmatrix*}[r]
            1 & 0 & 2 \\
           -1 & 1 & 0 \\
            0 & 3 & 1
        \end{bmatrix*}
    \end{equation*}
    and then calculating the characteristic polynomial $f(\lambda)$ of $p(X)$:
    \begin{equation*}
        A = X^2 - X - I =
        \begin{bmatrix*}[r]
            -1 &  6 &  2 \\
            -1 & -1 & -2 \\
            -3 &  3 & -1
        \end{bmatrix*}.
    \end{equation*}
    The roots of $f(\lambda)$, calculated using the  cubic formula, are:
    \begin{eqnarray*}
        \mu_1 & = & (\sqrt[3]{6})^2-\sqrt[3]{6}-1,\\
        \mu_2 & = & -1-\frac{1}{2}\left[(\sqrt[3]{6})^2-\sqrt[3]{6}\right]+\frac{1}{2}i\sqrt{3}\left[(\sqrt[3]{6})^2+\sqrt[3]{6}\right],\\
        \mu_3 & = & -1-\frac{1}{2}\left[(\sqrt[3]{6})^2-\sqrt[3]{6}\right]-\frac{1}{2}i\sqrt{3}\left[(\sqrt[3]{6})^2+\sqrt[3]{6}\right].
    \end{eqnarray*}
    The composition is $f(p(\lambda))=\lambda^6-3\lambda^5+3\lambda^4-\lambda^3+18\lambda^2-18\lambda-30$ and can be factored as $f(p(\lambda))=(\lambda^3-6)(\lambda^3-3\lambda^2+3\lambda+5)$. So, according to Theorem~\ref{Thm:main}, there should be a rational solution $X$ to $p(X)=A$. This is the case by construction. 
    
    The roots of $f(p(\lambda))$ are as follows:
    \begin{eqnarray*}
        \gamma_1=\sqrt[3]{6}; && \gamma_2=-\frac{1}{2}\sqrt[3]{6}-i\frac{\sqrt{3}}{2}\sqrt[3]{6};\qquad\quad \gamma_3=-\frac{1}{2}\sqrt[3]{6}+i\frac{\sqrt{3}}{2}\sqrt[3]{6};\\
        \gamma_4=-\sqrt[3]{6}+1; && \gamma_5=1+\frac{1}{2}\sqrt[3]{6}+i\frac{1}{2}\sqrt{3}\sqrt[3]{6};\quad \gamma_6=1+\frac{1}{2}\sqrt[3]{6}-i\frac{1}{2}\sqrt{3}\sqrt[3]{6}.
    \end{eqnarray*}
    When these roots are inserted into $p(\lambda)$, we obtain the following: $p(\gamma_1)=\mu_1=p(\gamma_4)$, $p(\gamma_2)=\mu_2=p(\gamma_5)$, $p(\gamma_3)=\mu_3=p(\gamma_6)$. The first three roots are the roots of $\lambda^3-6$, while the last three roots are the roots of $\lambda^3-3\lambda^2+3\lambda+5$, which are also the eigenvalues of $X$. 
    
    \bigskip
    
    It is easy to see that $\mathbb{Q}(\mu_1,\mu_2,\mu_3)=\mathbb{Q}(\sqrt[3]{6},i\sqrt{3})$. Thus, the automorphisms in the Galois group are completely determined by where they send $\sqrt[3]{6}$ and $i\sqrt{3}$. Now, since $\tau$ is a homomorphism that fixes $\mathbb{Q}$, we get the following:
    \begin{equation} \label{eq:GaloisElemDef1}
        6=\tau(6)=\tau(\sqrt[3]{6}\sqrt[3]{6}\sqrt[3]{6})=\tau(\sqrt[3]{6})\tau(\sqrt[3]{6})\tau(\sqrt[3]{6})=\left(\tau(\sqrt[3]{6})\right)^3
    \end{equation}
    and
    \begin{equation} \label{eq:GaloisElemDef2}
        -3 = \tau(-3) = \tau(i\sqrt{3}\,i\sqrt{3}) = \tau(i\sqrt{3})\tau(i\sqrt{3}) = \left(\tau(i\sqrt{3})\right)^2.
    \end{equation}
    This means that $\tau(\sqrt[3]{6})$ is a cube root of $6$ and $\tau(i\sqrt{3})$ is a square root of $-3$, so there are three options for where $\tau$ will send $\sqrt[3]{6}$ and two options for where $\tau$ will send $i\sqrt{3}$. 
    
    The Galois group $G={\mathit{Gal}}(\mathbb{Q}(\mu_1,\mu_2,\mu_3):\mathbb{Q})$ therefore has at most 6 elements and since the automorphisms in $G$ permute three elements ($\mu_1,\mu_2,\mu_3$), $G$ is a subgroup of $S_3$, the symmetric group on three letters.
    
    Define the elements in the Galois group $G$ according to \eqref{eq:GaloisElemDef1} and \eqref{eq:GaloisElemDef2}. Let $\tau_1$ be defined as the identity: $\tau_1(\sqrt[3]{6})=\sqrt[3]{6}$ and $\tau_1(i\sqrt{3})=i\sqrt{3}$; let $\tau_2$ and $\tau_3$ be defined as:
    \begin{equation*}
        \tau_2(\sqrt[3]{6}) = -\frac{1}{2}\sqrt[3]{6} - \frac{i\sqrt{3}}{2}\sqrt[3]{6}; \quad \quad \tau_2(i\sqrt{3}) = i\sqrt{3}
    \end{equation*}
    and
    \begin{equation*}
        \tau_3(\sqrt[3]{6}) = -\frac{1}{2}\sqrt[3]{6} + \frac{i\sqrt{3}}{2}\sqrt[3]{6}; \quad \quad \tau_3(i\sqrt{3}) = i\sqrt{3},
    \end{equation*}
    respectively. With the first three Galois group elements defined in this way, we have $\tau_1(\mu_1) = \mu_1$, $\tau_2(\mu_1) = \mu_2$ and $\tau_3(\mu_1) = \mu_3$. The remaining three elements are defined as follows:
    \begin{equation*}
        \tau_4(\sqrt[3]{6})=\sqrt[3]{6};\quad\quad \tau_4(i\sqrt{3})=-i\sqrt{3};
    \end{equation*} 
    \begin{equation*}
        \tau_5(\sqrt[3]{6}) = -\frac{1}{2}\sqrt[3]{6} - \frac{i\sqrt{3}}{2}\sqrt[3]{6}; \quad \quad \tau_5(i\sqrt{3}) = -i\sqrt{3};
    \end{equation*}
    and
    \begin{equation*}
        \tau_6(\sqrt[3]{6}) = -\frac{1}{2}\sqrt[3]{6} + \frac{i\sqrt{3}}{2}\sqrt[3]{6}; \quad \quad \tau_6(i\sqrt{3}) = -i\sqrt{3}.
    \end{equation*}
    The Galois element $\tau_4$ is complex conjugation. We have the following maps for $\mu_1$: $\tau_4(\mu_1) = \mu_1$; $\tau_5(\mu_1) = \mu_2$; $\tau_6(\mu_1) = \mu_3$. Note that $\mu_1$ (and $\gamma_1$) only depends on where $\sqrt[3]{6}$ is sent to.
    
    We now show that $\tau_2(\mu_1)=\mu_2$.
    \begin{eqnarray*}
        \tau_2(\mu_1)&=& \tau_2 \left((\sqrt[3]{6})^2-\sqrt[3]{6}-1\right)\\
        &=& (\tau_2(\sqrt[3]{6}))^2-\tau_2(\sqrt[3]{6})-1\\
        &=& (\sqrt[3]{6})^2\left(\frac{1}{4}+\frac{\sqrt{3}}{2}i-\frac{3}{4}\right)-\left(-\frac{1}{2}\sqrt[3]{6}-\frac{i\sqrt{3}}{2}\sqrt[3]{6}\right)-1\\
        &=& (\sqrt[3]{6})^2\left(-\frac{1}{2}+\frac{i\sqrt{3}}{2}\right)+\frac{1}{2}\sqrt[3]{6}+\frac{i\sqrt{3}}{2}\sqrt[3]{6}-1\\
        &=& -1-\frac{1}{2}\left[(\sqrt[3]{6})^2-\sqrt[3]{6}\right]+\frac{i\sqrt{3}}{2}\left[(\sqrt[3]{6})^2+\sqrt[3]{6}\right]\\
        &=& \mu_2.
    \end{eqnarray*}
    The fact that $\tau_2(\mu_2)=\mu_3$ and $\tau_2(\mu_3)=\mu_1$ follows similarly. Finally, one can easily see that if $\tau_i$ sends $\mu_1$ to $\mu_j$, then $\tau_i$ sends $\gamma_1$ to $\gamma_j$ for $1\leq i\leq 6$, $1\leq j\leq 3$. 
    
    Since $\gamma_i\in\mathbb{Q}(\sqrt[3]{6}, i\sqrt{3})$ for $i=1, \ldots , 6$, the splitting fields of $f(p(\lambda))$ and $f(\lambda)$ coincide, and hence also the Galois groups $Gal(\mathbb{Q}(\gamma_1, \ldots ,\gamma_6):\mathbb{Q})$ and $Gal(\mathbb{Q}(\mu_1,\mu_2,\mu_3):\mathbb{Q})$ coincide.\hfill$\Box$
\end{example}

The previous examples are a special case of the following proposition.

\begin{proposition}
 Suppose $p(\lambda)\in \mathbb{Q}[\lambda]$ is a monic quadratic polynomial, and that
$f(\lambda)\in\mathbb{Q}[\lambda]$ is irreducible and of degree $n$. Assume also that for every
root $\mu$ of $f(\lambda)=0$ there is a $\gamma\in\mathbb{Q}(\mu)$ such that $p(\gamma)=\mu$.
Then the splitting field of $f(p(\lambda))$ over $\mathbb{Q}$ is equal to the splitting field of $f(\lambda)$
over $\mathbb{Q}$.
\end{proposition}

\begin{proof}
Notice that for every complex number $\mu$ there are two solutions $\gamma_1, \gamma_2$
of $p(\lambda)=\mu$. Let $p(\lambda)=\lambda^2+p_1\lambda +p_0$. Then $\gamma_1\gamma_2=p_0-\mu$.
Hence, if $\gamma_1\in\mathbb{Q}(\mu)$, then also $\gamma_2\in\mathbb{Q}(\mu)$.

Now let $\mathbb{Q}(\mu_1, \ldots, \mu_n)$ be the splitting field of $f(\lambda)$ over $\mathbb{Q}$.
By assumption, for each $\mu_i$ there is at least one $\gamma_{i1}\in\mathbb{Q}(\mu_i)$ such that
$p(\gamma_{i1})=\mu_i$. As argued above, it follows that also the other solution $\gamma_{i2}$
must be in $\mathbb{Q}(\mu_i)$, and so both solutions are in $\mathbb{Q}(\mu_1, \ldots , \mu_n)$.
The roots of $f(p(\lambda))=0$ are given by the $2n$ solutions of $p(\lambda)=\mu_i$, for $i=1, \ldots , n$,
and therefore $\mathbb{Q}(\gamma_{11}, \gamma_{12}, \gamma_{21}, \gamma_{22}, \ldots ,
\gamma_{n1},\gamma_{n2})\subset \mathbb{Q}(\mu_1, \mu_2, \ldots , \mu_n)$.

Since $\mu_i=p(\gamma_{i1})$ for all $i=1, \ldots , n$ the other inclusion is evident.\end{proof}

It is clear that a similar argument will fail when $p(\lambda)$ is of degree larger than two.
This is illustrated in the following example.

\begin{example}
Let $A=\begin{bmatrix*}[r] 1 & -2 \\ -4 & 1 \end{bmatrix*}$. The characteristic polynomial of $A$ is
$f(\lambda)=(\lambda-1)^2-8$, which is irreducible. The splitting field of $f(\lambda)$ is $\mathbb{Q}(\sqrt{2})$; the eigenvalues of $A$ are $1\pm 2\sqrt{2}$. One checks directly that with $p(\lambda)=\lambda^3-4\lambda+1$, and $X=\begin{bmatrix} 0 & 1 \\ 2 & 0\end{bmatrix}$ we have $p(X)=A$. Then $f(p(\lambda))=(\lambda^3-4\lambda)^2-8=\lambda^6-8\lambda^4+16\lambda^2-8$, which factorizes as 
\[
f(p(\lambda))= (\lambda^2-2)(\lambda^4-6\lambda^2+4)=(\lambda^2-2)((\lambda^2-3)^2-5).
\]
Hence the six roots of $f(p(\lambda))=0$ are given by $\pm\sqrt{2} , \pm\sqrt{3\pm\sqrt{5}}$.
Thus the splitting field of $f(p(\lambda))$ is $\mathbb{Q}(\sqrt{2},\sqrt{3+\sqrt{5}})$.
\hfill $\Box$
\end{example}

\bigskip

\subsection*{Connection with a constructive approach of Drazin's paper.}

\bigskip
Next, we consider how the main result connects with a more constructive approach which originates in the paper by Drazin, \cite{Drazin}. %, and was further elaborated on in our earlier paper \cite{GJRST}.
The setting is the same as above: $A$ is an $n\times n$ matrix with entries in $\mathbb{Q}$, with an irreducible characteristic polynomial $f(\lambda)$, and $p(\lambda)$ is a polynomial with coefficients in $\mathbb{Q}$. We summarize the results of \cite{Drazin}.
If $p(X)=A$ has a solution $X$ with entries in $\mathbb{Q}$, for every eigenvalue $\mu_j$ of $A$ ($j=1,2,\ldots ,n$), there is at least one $\gamma_j\in\mathbb{Q}(\mu_j)$ such that $p(\gamma_j)=\mu_j$ (see \cite{Drazin}, Proposition 2.3).

Conversely, let $p(\gamma)=\mu$ for some eigenvalue $\mu$ of $A$ and some $\gamma\in\mathbb{Q}(\mu)$. Since $f(\lambda)$ is irreducible, $f(\lambda)$ is the minimum polynomial of $\mu$ over $\mathbb{Q}$ and hence $[\mathbb{Q}(\mu):\mathbb{Q}]=n$. Let $Aw=\mu w$, so $w$ is the eigenvector of $A$ corresponding to eigenvalue $\mu$. Then the entries of $w$ are in $\mathbb{Q}(\mu)$, so there is an $n\times n$ matrix $W$ with entries in $\mathbb{Q}$ such that $Wv_n(\mu)=w$, where $v_n(\mu)$ is the vector $v_n(\mu)=\begin{bmatrix} 1 & \mu & \mu^2 & \cdots & \mu^{n-1}\end{bmatrix}^T$. Now, following \cite{Drazin}, we can construct $X$ as follows. From the fact that we want $p(X)=A$ we derive that necessarily 
$Xw=\gamma w$. The latter equation can also be expressed as $XWv_n(\mu)=\gamma Wv_n(\mu)$.
Because of the fact that $\gamma\in\mathbb{Q}(\mu)$, there is an $n\times n$ matrix $C$ with rational entries such that $\gamma Wv_n(\mu)=Cv_n(\mu)$. It can be shown (see \cite{Drazin}) that $W$ is invertible, so solving $X$ from $XW=C$ produces an $n\times n$ matrix $X$ with rational entries such that $p(X)=A$.

\subsection*{Number of Solutions}
\bigskip

%There is one issue that was still bugging Madelein and me. Namely, we would like to say that the number of solutions of $p(X) = A$ is equal to the number of admissible $\gamma$'s 
%corresponding to $\mu_i$, for any $i$.

Introduce the following terminology:
a solution $\gamma$ of $p(\lambda) = \mu_i$ will be called \textit{admissible} if $\gamma \in \QQ(\mu_i)$, and two admissible solutions $\gamma_i$ and $\gamma_j$ are called \emph{$G$-connected} if there exists an element $g$ of $G$ such that $g(\gamma_i) = \gamma_j$. 

 %Usually, we start the procedure by looking at the $\gamma$s associated to $\mu_1$, and then act with the Galois group to produce admissible $\gamma$s associated to the remaining eigenvalues of $A$. However, in principle, one can begin the analysis at any of the eigenvalues. This raises the following questions:

\begin{proposition}\label{numberofsolutions}
Let $A$ be an $n\times n$ matrix over $\mathbb{Q}$ with irreducible characteristic polynomial $f(\lambda)$, and let $p(\lambda)$ be any polynomial over $\mathbb{Q}$. Then for every eigenvalue $\mu$ of $A$ the number of admissible $\gamma$'s is the same, and this number equals the number of rational solutions to $p(X)=A$.
\end{proposition}

\begin{proof}Let $\gamma_i$ be an admissible solution of $p(\lambda) = \mu_i$ for $1 \leq i \leq n$.  We use a set of $n$ $G$-connected admissible elements to construct the factor $h(\lambda)$. It cannot occur that two distinct admissible elements $\gamma_i^1$ and $\gamma_i^2$ associated to $\mu_i$ are both $G$-connected to the same admissible element $\gamma_j$ associated to $\mu_j$.  
Indeed, if $g \in G$ fixes some eigenvalue $\mu_i$, then the restriction $g|_{\QQ(\mu_i)}$ must be the identity, and hence $g$ must fix each admissible element associated to $\mu_i$ since by definition they all lie inside $\QQ(\mu_i)$. From this it also follows that 
the number of admissible $\gamma$'s associated to an eigenvalue $\mu_i$ is the same for every eigenvalue $\mu_i$.

Now the number of rational solutions to $p(X)=A$ is equal to the number of admissible solutions of $p(\lambda)=\mu_i$ (see \cite{Drazin}).
Therefore, the number of solutions of $p(X) = A$ is equal to the number of admissible elements associated to any eigenvalue.
\end{proof}
\bigskip

\section{The simple case}

\medskip

By using the idea of working separately with the irreducible parts of the characteristic polynomial of a simple matrix and using the companion-Jordan form of the matrix, we can prove the following result.
Recall that a matrix is called \emph{simple} if the algebraic multiplicity of each eigenvalue is one.
In particular, when the characteristic polynomial of the matrix is irreducible, then the matrix is simple, but the converse is not true.

\begin{proposition}
Let $\mathbb{F}$ be a field such that $\mathbb{Q}\subset \mathbb{F}\subset \mathbb{C}$ and $A\in M_n(\mathbb{F})$ be a simple matrix. Let the characteristic polynomial of $A$ be\\ $f(\lambda)=f_1(\lambda)f_2(\lambda)\cdots f_r(\lambda)$, for some $r$, where $f_i(\lambda)$ are distinct, irreducible and of degree $n_i$. Let $p(\lambda)\in\mathbb{F}[\lambda]$. Then $p(X)=A$ has a solution $B\in M_n(\mathbb{F})$ if and only if $f_i(p(\lambda))$ has a factor of degree $n_i$ in $\mathbb{F}[\lambda]$, for each $1 \leq i \leq r$.
\end{proposition}

\begin{proof}
Let $A\in M_n(\mathbb{F})$ be simple with characteristic polynomial $f(\lambda)=f_1(\lambda)f_2(\lambda)\cdots f_r(\lambda)$, where $f_i(\lambda)$ is irreducible and of degree $n_i$. By the companion-Jordan form (see \cite{GJRST2}, Theorem 2.1, also \cite{robinson}), there exists an invertible matrix $T\in M_n(\mathbb{F})$ such that 
\begin{equation}\label{eqGeneralSimpleA}
A=T^{-1}\begin{bmatrix}
C_1 &&0\\
&\ddots&\\
0&&C_r
\end{bmatrix} T,
\end{equation}
where $C_i$ is the $n_i\times n_i$ companion matrix of the polynomial $f_i(\lambda)$. 

Suppose $p(X)=A$ has a solution $B\in M_n(\mathbb{F})$. Then by Proposition 5.1 in \cite{GJRST2},  $B$ is of the form 
\begin{equation*}
B=T^{-1}\begin{bmatrix}
B_1 &&0\\
&\ddots&\\
0&&B_r
\end{bmatrix} T,
\end{equation*}
where the sizes of $B_i$ correspond to those of $C_i$. Now, from $p(B)=A$ we  obtain $r$ different equations $p(B_i)=C_i$, $1\leq i\leq r$, since we can write 
\begin{align*}
T^{-1}p\left(B_1\oplus\cdots\oplus B_r\right)T
&= p\left(T^{-1}(B_1\oplus\cdots\oplus B_r)T\right)=p(B)=A\\ &= T^{-1}\left(C_1\oplus\cdots\oplus C_r\right)T.
\end{align*}
It is easy to check the facts that $p(X_1\oplus\cdots\oplus X_r)=p(X_1)\oplus\cdots\oplus p(X_r)$ and $T^{-1}p(X)T=p(T^{-1}XT)$ for any polynomial $p(\lambda)$ and square matrices $X_i$ and $X$.
 Remember that $C_i$ has irreducible characteristic polynomial $f_i(\lambda)$. Therefore, by Proposition 1 of \cite{reams}, or our main theorem above, $f_i(p(\lambda))$ has a factor of degree $n_i$ in $\mathbb{F}[\lambda]$ for all $1\leq i\leq r$.

Conversely, let $f_i(p(\lambda))$ have a factor of degree $n_i$ in $\mathbb{F}[\lambda]$ for all $1\leq i\leq r$. Then, again by Proposition 1 of \cite{reams} or the main theorem above, $p(X)=C_i$ has a solution with $C_i$ as in \eqref{eqGeneralSimpleA}, that is, there is a $B_i\in M_{n_i}(\mathbb{F})$ such that $p(B_i)=C_i$. Hence, we have $p(B_1\oplus\cdots\oplus B_r)=C_1\oplus\cdots\oplus C_r$ and then 
\begin{align*}
p(B)&= p\left(T^{-1}(B_1\oplus\cdots\oplus B_r) T\right) \\
&= T^{-1}p\left(B_1\oplus\cdots\oplus B_r\right) T \\
&= T^{-1}\left(C_1\oplus\cdots\oplus C_r\right)T = A.
\end{align*}
Thus, $p(X)=A$ has a solution $B$ in $M_n(\mathbb{F})$.
\end{proof}

\section{The nonderogatory case}

Let $A$ be nonderogatory, and assume that the characteristic polynomial is given by $f(\lambda)=
f_1(\lambda)^{d_1}\cdot f_2(\lambda)^{d_2} \cdots f_r(\lambda)^{d_r}$ with the $f_i(\lambda)$'s pairwise coprime and irreducible and of degree $k_j$. We follow the construction of Theorem 7.1 in \cite{GJRST2}. Let $p(\lambda)=\sum_{i=0}^l p_i\lambda^i \in\mathbb{Q}[\lambda]$. 

Assume that for each eigenvalue $\mu$ of $A$ there is a solution
$\gamma\in \mathbb{Q}(\mu)$ of $p(\gamma)=\mu$. If $\mu$ is a root of $f_j(\lambda)$, let $G_j$ be the Galois group of $f_j(\lambda)$, and consider the factor $h_j(\lambda)$ of degree $k_j$ of $f_j(p(\lambda))$ which we obtain using the explicit construction above, that is, we consider the action of $G_j$ on $\gamma$ to obtain the $k_j$ roots of $h_j(\lambda)$. Denote these roots by $\gamma_{j,1}, \ldots , \gamma_{j,k_j}$. Now assume that for each $j$ for which $d_j>1$, there is at least one $\gamma\in\mathbb{Q}(\mu_j)$ such that for each pair of roots $\gamma_{j,s}$ and $\gamma_{j,t}$ we have that condition (13) in \cite{GJRST2} is satisfied, that is:
\[
\sum_{m=1}^l p_m \sum_{i=0}^{m-1} \gamma_{j,s}^i\gamma_{j,t}^{m-1-i}\not=0.
\]
Then, according to Theorem 7.1, part (iii) in \cite{GJRST2} there is a rational solution $X$ to $p(X)=A$. 
This gives a completely algebraic sufficient condition for the existence of a rational solution.
However, the condition is not necessary, as pointed out in \cite{GJRST2}.

\medskip

{\bf Acknowledgement.} This work is based on research supported  in part by the National Research Foundation of
South Africa  (Grant Number 145688). Opinions expressed and conclusions arrived 
at are those of the authors and are not necessarily to be attributed to the NRF.


\begin{thebibliography}{xx}

\bibitem{Bewersdorff}
J. Bewersdorff. \emph{Galois Theory for Beginners, A Historical Perspective}. Amer. Math. Soc. 2006.

\bibitem{Drazin} M.P.~Drazin. Exact rational solutions of the matrix equation $A=p(X)$ by linearization. \emph{Linear Algebra Appl.}, {\bf 426} (2007) 502--515. 

\bibitem{EU}
J-C.~Evard and F.~Uhlig.
On the matrix equation $f(X)=A$.
\emph{Linear Algebra Appl.}, {\bf 162--164} (1992), 447--519.

\bibitem{FI1}
F.~Fasi and B.~Iannazzo. Computing primary solutions of equations involving primary matrix functions. \emph{Linear Algebra Appl.}, {\bf 560} (2019), 17--42.

\bibitem{FI2}
F.~Fasi and B.~Iannazzo. Substitution algorithms for rational matrix equations. \emph{Electron. Trans. Numer. Anal.}, {\bf 53} (2020), 500--521.

\bibitem{FI3}
F.~Fasi and B.~Iannazzo. The dual inverse scaling and squaring algorithm for the matrix logarithm. \emph{IMA J. Numer. Anal.}, {\bf 42} (2022), 2829--2851.

\bibitem{GJRST}
G.J. Groenewald, D.B. Janse van Rensburg, A.C.M. Ran, F. Theron, M. van Straaten.
\newblock{$m$th roots of $H$-selfadjoint matrices.}
\newblock{\it Linear Algebra Appl.} 610 (2021), 804--826.

\bibitem{GJRST2}
G.J. Groenewald, D.B. Janse van Rensburg, A.C.M. Ran, F. Theron, M. van Straaten.
\newblock{The solutions of the matrix equation $p(X)=A$, with polynomial function $p(\lambda)$ over field extensions of $\mathbb{Q}$.}
\newblock{\it Linear Algebra Appl.} 665 (2023), 107--138.

\bibitem{higham}
N.J.~Higham. 
\newblock{\it Functions of Matrices: Theory and Computation},
\newblock{Society for Industrial and Applied Mathematics, Philadelphia, PA, USA}, 2008.

\bibitem{otero} D.E.~Otero. Extraction of $m$th roots in matrix rings over fields.
{\it Linear Algebra Appl.} {\bf 128} (1990), 1--26.

\bibitem{pinter} C.C.~Pinter \emph{A book of Abstract Algebra}, Second edition, Dover Publ., New York, 1990.

\bibitem{psarrakos}
{P.J.~Psarrakos}.
\newblock {On the $m$th roots of a complex matrix},
\newblock {\it Electron. J. Linear Algebra}, {\bf 9} (2002), 32--41.

\bibitem{reams}
R.~Reams.
A Galois approach to $m$th roots of matrices with rational entries.
{\it Linear Algebra Appl.}, {\bf 258} (1997), 187--194.

\bibitem{robinson}
D.W.~Robinson. The generalized Jordan canonical form. \emph{Amer. Math. Monthly}, {\bf 77} (1970), 392--395.

\bibitem{roman}
S.~Roman.
\emph{Advanced Linear Algebra}, Third edition,
Springer, New York, 2008. 

\bibitem{roth}
W.E.~Roth. A solution of the matric equation $P(X)=A$. 
\emph{Transact. Amer. Math. Soc.} {\bf 30} (1928), 579--596.

\bibitem{Spiegel}
E. Spiegel. On the matrix roots of $f(X)=A$.
\emph{Indian J. Pure Appl. Math.} {\bf 19} (1988), 854--864.

\bibitem{Stewart}
I. Stewart, \emph{Galois Theory}, Second edition, Chapman \& Hall 1989.

\bibitem{tenhave}
{G.~ten Have}.
{Structure of the $n$th roots of a matrix},
{\it Linear Algebra Appl.}, {\bf 187} (1993), 59--66.

\bibitem{wedderburn}
{J.H.M.~Wedderburn}.
{\it Lectures on Matrices},
{American Mathematical Society, New York}, 1934.
\end{thebibliography}
\end{document}